\documentclass{article}
\usepackage{amsmath,amssymb,amsfonts,amsthm,bm}
 \usepackage{enumerate}
 
\newtheorem{theorem}{Theorem}

\newtheorem{assumption}[theorem]{Assumption}

\DeclareMathOperator{\vvec}{vec}
\newcommand{\wh}{\widehat}
\newcommand{\wt}{\widetilde}
\newcommand{\s}{\mathcal{S}}
\newcommand{\ZZ}{\mathbb{Z}}

\newcommand{\xr}{{\bm r}}
\newcommand{\xs}{{\bm s}}
\newcommand{\xd}{{\bm d}}
\newcommand{\xl}{{\bm \ell}}

\DeclareFontFamily{OT1}{pzc}{}
\DeclareFontShape{OT1}{pzc}{m}{it}%
              {<-> s * [0.900] pzcmi7t}{}
\DeclareMathAlphabet{\mathpzc}{OT1}{pzc}%
                                 {m}{it}
\begin{document}

\title{Shift techniques for Quasi-Birth and Death processes: canonical factorizations and matrix equations\footnote{This work has been supported by GNCS of INdAM and by the PRA project ``Mathematical models and computational methods for complex networks'' at University of Pisa.}}
\author{D.A. Bini, Universit\`a di Pisa, Italy,
\\[1ex]
G. Latouche, Universit\'e libre de Bruxelles, Belgium,
\\[1ex]
B. Meini,Universit\`a di Pisa, Italy.}

\maketitle

\begin{abstract}
We revisit the shift technique applied to Quasi-Birth and Death (QBD)
processes (He, Meini, Rhee, SIAM J. Matrix Anal. Appl., 2001) by
bringing the attention to the existence and properties of canonical
factorizations. To this regard, we prove new results concerning the
solutions of the quadratic matrix equations associated with the QBD.
These results find applications to the solution of the Poisson
equation for QBDs.
\end{abstract}
{\bf Keywords:}
 Quasi-Birth-and-Death processes, Shift technique, Canonical factorizations, Quadratic matrix equations

\noindent {\bf AMS Subject Calssification:}
   15A24,
    47A68,
    60J22.  

\section{Introduction}
Quadratic matrix equations of the kind
\begin{equation}\label{eq:uqme}
A_{-1}+(A_0-I)X+A_1X^2=0,
\end{equation}
where $A_{-1},A_0,A_1$ are given $n\times n$ matrices, are encountered
in many applications, say in the solution of the quadratic eigenvalue
problem, like vibration analysis, electric circuits, control theory
and more \cite{tm,hk}. In the area of Markov chains, an important
application concerns the solution of Quasi-Birth-and-Death (QBD)
stochastic processes, where it is assumed that $A_{-1}$, $A_0$ and
$A_1$ are nonnegative matrices such that $A_{-1}+A_0+A_1$ is
stochastic and irreducible \cite{lr99,blm:book}.

 For this class of problems, together with \eqref{eq:uqme}, the
 dual equation
$X^2A_{-1}+X(A_0-I)+A_1=0$ has a relevant interest.
It is well known that both \eqref{eq:uqme} and the dual equation
have minimal nonnegative solutions $G$ and $R$,
respectively, according to the component-wise ordering, which can be
explicitly related to one another \cite{lr99,neuts81}. These solutions
have an interesting probabilistic interpretation and their
computation is a fundamental task in the analysis of QBD processes.
Moreover they provide the factorization $\varphi(z)=(I-zR)K(I-z^{-1}G)$ of the Laurent polynomial $\varphi(z)=z^{-1}A_{-1}+A_0-I+zA_1$, where $K$ is a nonsingular matrix. A factorization of this kind is canonical if $\rho(R)<1$ and $\rho(G)<1$, where $\rho$ denotes the spectral radius. It is said weak  canonical if $\rho(R)\le 1$ and $\rho(G)\le 1$.

We introduce the matrix polynomial $B(z)=A_{-1}+z(A_0-I)+z^2A_1=z\varphi(z)$ and define the roots of
$B(z)$ as the
zeros of the polynomial $\det B(z)$. If $\xi$ is a root of $B(z)$ we
say that $v$ is an eigenvector associated with $\xi$ if $v\ne 0$ and
$B(\xi)v=0$. 
The location of the roots of $B(z)$
determines the classification of the QBD as positive,
null recurrent or transient, 
and governs the convergence
and the efficiency of the available numerical algorithms for
approximating $G$ and $R$ \cite{blm:book}.
In particular, $B(z)$ has always a root on the unit circle, namely, the
root $\xi=1$, and the corresponding eigenvector is the vector $e$ of
all ones, i.e., $B(1)e=0$.  

If the QBD is recurrent, the root $\xi=1$ is the eigenvalue of largest
modulus of the matrix $G$ and $Ge=e$. In the transient case, that root
is the eigenvalue of  largest modulus of $R$.  These facts have
been used to improve convergence properties of numerical methods for
computing the matrix $G$. The idea, introduced in \cite{hmr01} and
based on the results of \cite{brauer}, is to ``shift'' the root
$\xi=1$ of $B(z)$ to zero or to infinity, and to construct a new
quadratic matrix polynomial $B_\xs(z)=A_{-1}^\xs+z(A_0^\xs-I)+z^2
A_1^\xs$ having the same roots as $B(z)$, except for the root equal to 1,
which is replaced with 0 or infinity. Here the super-(sub-)script $\xs$ means ``shifted''. 
This idea has been subsequently
developed and applied in
\cite{bms:mam5,guo03,gh:shift,gim:shift,ip:shift,meini:ilas10}.

In this paper we revisit the shift technique, and we focus on the
properties of the canonical factorizations.  In particular, we prove
new results concerning the existence and properties of the solutions
of the quadratic matrix equations obtained after the shift.

  By following \cite{blm:book}, we recall that in the positive
  recurrent case the root $\xi=1$ can be shifted to zero by
  multiplying $B(z)$ to the right by a suitable function (right
  shift), while in the transient case the root $\xi=1$ can be shifted
  to infinity by multiplying $B(z)$ to the left by another suitable
  function (left shift).  In the null recurrent case, where $\xi=1$ is
  a root of multiplicity 2, shift is applied both to the left and to
  the right so that one root 1 is shifted to zero and the other root 1
  is shifted to infinity (double shift).  In all the cases, the new
  Laurent matrix polynomial $\varphi_\xs(z)=z^{-1}B_\xs(z)$ is
  invertible on an annulus containing the unit circle in the complex
  plane and we prove that it admits a canonical factorization which is
  related to the weak canonical factorization of
  $\varphi(z)$.  As a consequence, we relate $G$ and $R$ with the
  solutions $ G_{\xs}$ and $R_{\xs}$ of minimal spectral radius of the
  matrix equations $A_{-1}^\xs+(A_0^\xs-I) X+ A_1^\xs X^2=0$ and $X^2
  A_{-1}^\xs+X (A_0^\xs-I)+A_1^\xs=0$, respectively.

A less trivial issue is the existence of the canonical factorization
of $\varphi_\xs(z^{-1})$. We show that such factorization exists and we provide an
explicit expression for it, for the three different kinds of
shifts. The existence of such factorization allows us to express the
minimal nonnegative solutions $\wh G$ and $\wh R$ of the matrix
equations $A_{-1}X^2+(A_0-I)X+ A_{1}=0$ and $A_{-1}+X (A_0-I)+ X^2 A_{1}=0$,
in terms of the solutions of minimal spectral radius $\wh G_{\xs}$
and $\wh R_{\xs}$ of the equations $ A_{-1}^\xs X^2+ (A_0^\xs-I) X+ A_{1}^\xs=0$
and $A_{-1}^\xs+X (A_0^\xs-I) + X^2 A_{1}^\xs=0$, respectively.

The existence of the canonical factorizations of 
$\varphi_\xs(z)$ and $\varphi_\xs(z^{-1})$ has interesting
consequences. Besides providing computational advantages in the
numerical solution of matrix equations, it allows one to give an explicit
expression for the solution of the Poisson problem for QBDs
\cite{bdlm:poisson}.  Another interesting issue related to the shift
technique concerns conditioning. In fact, while null recurrent
problems are ill-conditioned, the shifted counterparts are not. A
convenient computational strategy to solve a null recurrent problem
consists in transforming it into a new one, say by means of the double
shift; solve the latter by using a quadratic convergent algorithm like
cyclic reduction or logarithmic reduction~\cite{blm:book}; then
recover the solution of the original problem from the one of the
shifted problem. For this conversion, the expressions relating the
solutions of the shifted equations to those of the original equations
are fundamental, they are provided in this paper.

The paper is organized as follows. In Section \ref{sec:prel} we recall
some properties of the canonical factorization of matrix polynomials,
and their interplay with the solutions of the associated quadratic
matrix equations, with specific attention to those equations
encountered in QBD processes. In Section \ref{sec:shift} we present
the shift techniques in functional form, with attention to the
properties of the roots of the original and modified matrix
polynomial. In Section \ref{sec:cf} we state the main results on the
existence and properties of canonical factorizations. In particular we
provide explicit relations between the solutions of the original
matrix equations and the solutions of the shifted equations. 
In the Appendix, the reader can find the proof of a technical property used
to prove the main results.

\section{Preliminaries}\label{sec:prel}

In this section we recall some properties of matrix polynomials and of
QBDs, that will be used later in the paper. For a general treatment on
these topics we refer to the books
\cite{blm:book,glr82,nick:book,lr99,neuts81}.

\subsection{Matrix polynomials}
Consider the matrix Laurent polynomial
$\varphi(z)=\sum_{i=-1}^{1}z^i B_i$, where $B_i$,
$i=-1,0,1$, are $n\times n$ complex matrices.  A
canonical factorization of $\varphi(z)$ is a decomposition of the kind
$\varphi(z)=E(z)F(z^{-1})$,
where $E(z)=E_0+zE_1$ and 
$F(z)=F_0+zF_{-1}$
are invertible for $|z|\le 1$. 
A canonical factorization 
is  {\em weak} 
if $E(z)$ and $F(z)$ are invertible for $|z|< 1$ but
possibly singular for some values of $z$ such that $|z|=1$.
The canonical factorization is unique in the form $\varphi(z)=(I-z\wt E_1)K(I-z^{-1}\wt F_{-1})$ for suitable matrices $\wt E_1$, $\wt F_{-1}$ and $K$, see for instance \cite{cg:book}.

Given an $n\times n$ quadratic matrix polynomial
$B(z)=B_{-1}+zB_0+z^2B_1$, we call \emph{roots of} $B(z)$ the roots
$\xi_1,\ldots,\xi_{2n}$ of the polynomial $\det B(z)$ where we assume
that there are $k$ roots at infinity if the degree of $\det B(z)$ is
$2n-k$. In the sequel we also assume that the roots are ordered so
that $|\xi_1|\le \cdots\le|\xi_{2n}|$.

Consider the following matrix equations
\begin{eqnarray}
&B_{-1}+B_0X+B_1X^2=0,\label{eq:g}\\
&X^2 B_{-1}+XB_0+B_1=0,\label{eq:r}\\
&B_{-1}X^2+B_0X+ B_{1}=0,\label{eq:gh}\\
&B_{-1}+X B_0+ X^2 B_{1}=0.\label{eq:rh}
\end{eqnarray}

Observe that if $X$ is a solution of \eqref{eq:g} and $Xv=\lambda v$
for some $v\ne 0$, then $B(\lambda)v=0$ that is, $\lambda$ is a root
of $B(z)$. Similarly, the eigenvalues of any solution of 
\eqref{eq:rh} are roots of $B(z)$, and the reciprocal of the
eigenvalues of any solution of \eqref{eq:r} or \eqref{eq:gh} are roots
of $B(z)$ as well. Here we adopt the convention that $1/0=\infty$ and
$1/\infty=0$.

We state the following general result on canonical factorizations
which extends Theorem 3.20 of \cite{blm:book}:

\begin{theorem}\label{thm:h0}
  Let $\varphi(z)=z^{-1}B_{-1}+B_0+zB_1$ be an $n\times n$ Laurent matrix
  polynomial. Assume that the roots of $B(z)=z\varphi(z)$ are such that
  $|\xi_n|<1<|\xi_{n+1}|$ and that there exists a matrix $G$ which
  solves the matrix equation \eqref{eq:g} with $\rho(G)=|\xi_n|$. Then
  the following properties hold:
\begin{enumerate}
\item $\varphi(z)$ has the canonical factorization
  $\varphi(z)=(I-zR)K(I-z^{-1}G)$,  where $K=B_0+B_1G$,
  $R=-B_1K^{-1}$, $\rho(R)=1/|\xi_{n+1}|$ and $R$ is the solution of
  the equation \eqref{eq:r} with
  minimal spectral radius.

\item $\varphi(z)$ is invertible in the annulus $\mathbb A=\{
  z\in\mathbb C : |\xi_n|<z<|\xi_{n+1}|\}$ and
  $H(z)=\varphi(z)^{-1}=\sum_{i=-\infty}^{+\infty}z^iH_i$ is
  convergent for $z\in \mathbb A$, where
  \[
H_i=\left\{\begin{array}{ll}
G^{-i}H_0,& {\mbox for~} i<0,\\
\sum_{j=0}^{+\infty}G^j K^{-1} R^j,& {\mbox for~} i=0,\\
H_0R^i,& {\mbox for~} i>0.
\end{array}\right. 
\]

\item If $H_0$ is nonsingular, then $\varphi(z^{-1})$ has the
  canonical factorization $\varphi(z^{-1})=(I-z\wh R)\wh K(I-z^{-1}\wh
  G)$,  where $\wh K=B_0+B_1\wh G=B_0+\wh R B_{-1}$ and
  $\wh G=H_0 R H_0^{-1}$, $\wh R=H_0^{-1} G H_0$. Moreover, $\wh G$
  and $\wh R$ are the solutions of minimal spectral radius of the
  equations \eqref{eq:gh} and \eqref{eq:rh},
  respectively.
\end{enumerate}

\end{theorem}

\begin{proof}
Parts 1~and 2~are stated in Theorem 3.20 of \cite{blm:book}. We
prove Part 3. From 2, the function
$H(z^{-1})=\varphi(z^{-1})^{-1}$ is analytic in an annulus
$\wh{\mathbb A}$ contained in $\mathbb A$ and containing the unit circle.
From the expression of $H_i$ we
obtain
$ 
H(z^{-1})=\sum_{i=1}^{+\infty}z^{-i}H_0R^{i}+H_0+\sum_{i=1}^{+\infty}z^i G^i H_0.
$ 
Since $\det H_0\ne 0$, we may rewrite the latter equation as
\[
\begin{split}H(z^{-1})
&=\sum_{i=1}^{+\infty}z^{-i}(H_0R^{i}H_0^{-1})H_0+H_0+\sum_{i=1}^{+\infty}z^i H_0 (H_0^{-1}G^i H_0) \\
&= \sum_{i=1}^{+\infty}z^{-i}\wh G^{i}H_0+H_0+\sum_{i=1}^{+\infty}z^i H_0 \wh R^i,
\end{split}
\]
where we have set $\wh G=H_0RH_0^{-1}$ and $\wh R=H_0^{-1}G H_0$. Since the matrix power series in the above equation are convergent in $\wh{\mathbb A}$, 
more precisely $\sum_{i=1}^{+\infty}z^{-i}\wh G^{i}=(I-z^{-1}\wh G)^{-1}-I$ and
$\sum_{i=1}^{+\infty}z^i  \wh R^i=(I-z\wh R)^{-1}-I$, we may write
\[
\begin{split}
H(z^{-1})  
&=( (I-z^{-1}\wh G)^{-1}-I) H_0+H_0+H_0( (I-z\wh R)^{-1}-I)\\
&= (I-z^{-1}\wh G)^{-1} Y (I-z\wh R)^{-1}
\end{split}
\]
where 
\[
Y = H_0(I-z\wh R) - (I-z^{-1}\wh G)H_0(I-z\wh R) +(I-z^{-1}\wh G)H_0=
H_0-\wh G H_0 \wh R.
\]
The matrix $Y$ cannot be singular since otherwise $\det H(z^{-1})=0$ for any $z\in{\wh {\mathbb A}}$, which contradicts the invertibility of $H(z^{-1})$.
Therefore, we find that
$\varphi(z^{-1})=(I-z\wh R) Y^{-1} (I-z^{-1}\wh G)$ for $z\in\wh{\mathbb A}$, in particular for $|z|=1$. This factorization is canonical since $\rho(\wh R)=\rho(G)<1$ and  $\rho(\wh G)=\rho(R)<1$.
By the uniqueness of canonical factorizations \cite{cg:book}, one has
$Y^{-1}=\wh K=B_0+B_{-1}\wh G$. One finds, by direct inspection, that the matrices $\wh G$ and $\wh R$ are solutions  of
\eqref{eq:gh} and \eqref{eq:rh},  respectively. Moreover, they are solutions of minimal spectral radius since their eigenvalues coincide with the $n$ roots with smallest modulus of $B(z)$ and of $zB(z^{-1})$, respectively.
\end{proof}

The following result holds under weaker assumptions and provides the converse property of part 3 of Theorem \ref{thm:h0}.

\begin{theorem}\label{thm:w}
Let $\varphi(z)=z^{-1}B_{-1}+B_0+zB_1$ be an $n\times n$ Laurent matrix
  polynomial such that the roots of $B(z)=z\varphi(z)$ satisfy  $|\xi_n|\le 1\le|\xi_{n+1}|$.
The following properties hold:
\begin{enumerate}
\item 
If there exists a solution $G$ to the matrix equations 
\eqref{eq:g} such that $\rho(G)=|\xi_n|$, then
$\varphi(z)$ has the (weak) canonical factorization
  $\varphi(z)=(I-zR)K(I-z^{-1}G)$,  where $K=B_0+B_1G=B_0+RB_{-1}$,
  $R=-B_1K^{-1}$,  and $R$ is a solution of
   \eqref{eq:r} with $\rho(R)=1/|\xi_{n+1}|$;
  \item 
if there exists a solution $\wh G$ to the matrix equation
\eqref{eq:gh} such that $\rho(\wh G)=|1/\xi_{n+1}|$, then
$\varphi(z^{-1})$ has the (weak) 
  canonical factorization $\varphi(z^{-1})=(I-z\wh R)\wh K(I-z^{-1}\wh
  G)$,  where $\wh K=B_0+B_1\wh G=B_0+\wh R B_{-1}$, and
  $\wh R$ is a solution of
  \eqref{eq:rh} with $\rho(\wh R)=|\xi_n|$;
\item if $|\xi_n|<|\xi_{n+1}|$, and if there exist solutions $G$ and $\wh G$ to the matrix equations 
\eqref{eq:g} and \eqref{eq:gh}, respectively, such that $\rho(G)=|\xi_n|$, $\rho(\wh G)=|1/\xi_{n+1}|$, then the series 
$
W=\sum_{i=0}^\infty G^iK^{-1}R^i
$
 is convergent,
 $W$ is the unique solution of the Stein equation $X-GXR=K^{-1}$, 
  $W$ is nonsingular and $\wh G=W R W^{-1}$, $\wh R=W^{-1} G W$.
Moreover,  $W^{-1}=K(I-G\wh G)$ and $I-G\wh G$ is invertible.
\end{enumerate}
\end{theorem}

\begin{proof}
Properties 1 and 2 can be proved as Property 1 in Theorem 3.20 of \cite{blm:book}. 
Assume that $|\xi_n|<|\xi_{n+1}|$.
Since $\rho(G)$ or $\rho(R)$ is less than one, the series $\sum_{i=1}^\infty G^iK^{-1}R^i$ is convergent.
Observe that $GWR=\sum_{i=1}^\infty G^iK^{-1}R^i$ so that $W-GWR=K^{-1}$.
Therefore  $W$ solves the Stein equation  $X-GXR=K^{-1}$. The solution
is unique since $X$ solves the Stein equation if and only if
$(I\otimes I-R^T\otimes G)\vvec(X)=\vvec(K^{-1})$, where
$\vvec(\cdot)$ is the operator that stacks the columns of a matrix and
$\otimes$ is the Kronecker product; the matrix of the latter system is
nonsingular since $\rho(R^T\otimes G)=\rho(G)\rho(R)<1$. 
We prove that $\det (W)\ne 0$.
 Assume that $|\xi_n|\le 1$ and $|\xi_{n+1}|>1$  and choose
 $t\in\mathbb{R}$ such that $|\xi_n|<t<|\xi_{n+1}|$. Consider the
 matrix polynomial 
\[
B_t(z):=B(tz)=B_{-1,t}+zB_{0,t}+z^2B_{1,t},
\] where
 $B_{-1,t}=B_{-1}$, $B_{0,t}=tB_{0}$ and $B_{1,t}=t^2B_{1}$. The roots
 of $B_t(z)$ are $\xi_{i,t}=\xi_i/t$, $i=1,\ldots,2n$. Therefore, for
 the chosen $t$, we have 
$|\xi_{n,t}|<1<|\xi_{n+1,t}|$. Moreover the matrices $G_t=t^{-1}G$ and $\wh G_t=t\wh G$ 
are solutions with spectral radius less than one of the matrix equations
$B_{-1,t}+B_{0,t}X+B_{1,t}X^2=0$ and $B_{-1,t}X^2+B_{0,t}X+ B_{1,t}=0$, respectively.
In this way, the matrix polynomial $B_t(z)$ satisfies the assumptions
of Theorem 3.20 of \cite{blm:book}, and the matrix
$H_{0,t}=\sum_{i=0}^\infty G_t^i (B_{0,t}+B_{1,t}G_t)^{-1} R_t^i$,
where $R_t=tR$, is nonsingular. One  verifies by direct inspection
that $W=tH_{0,t}$. Therefore we conclude that $W$ is nonsingular as
well. 
Applying again Theorem 3.20 yields $\wh G_t=H_{0,t} R_t H_{0,t}^{-1}$
and $\wh R_t=H_{0,t}^{-1} G_t H_{0,t}$, where $\wh R_t=t^{-1}\wh R$,
therefore $\wh G=W R W^{-1}$ and $\wh R=W^{-1} G W$. Similar arguments
may be used if $|\xi_n|<1$ and $|\xi_{n+1}|\ge 1$.
Concerning the expression of $W^{-1}$, by the definition of $W$ we have
$
(I-G\wh G)W=W-\sum_{j=0}^\infty G^{j+1}K^{-1}R^{j+1}=K^{-1},
$
so that $W^{-1}=K(I-G\wh G)$.
\end{proof}

\subsection{Nonnegative matrices, quadratic matrix equations and QBDs}
A real matrix $A$ is nonnnegative (positive) if all its entries are nonnegative (positive), and we write $A\ge 0$ ($A>0$). If $A$ and $B$ are real matrices, we write $A\ge B$ ($A>B$) if $A-B\ge 0$ ($A-B>0$).
An $n\times n$ real matrix $M=\alpha I- N$ is called an M-matrix if $N\ge 0$  and $\alpha\ge\rho(N)$.
A useful property is that the inverse of a nonsingular M-matrix is nonnegative. 
For more properties on nonnegative matrices and M-matrices we refer to the book \cite{ber-ple:book}.

Assume we are given $n\times n$ nonnegative matrices $A_{-1}$, $A_0$ and $A_1$ such that $A_{-1}+A_0+A_1$ is stochastic.  The matrices $A_{-1}$, $A_0$ and $A_1$ define the homogeneous part of the
infinite transition matrix
\[
P = \left[
\begin{array}{cccc}
A'_0    & A'_1    &            & 0 \\
A_{-1} & A_0    & A_1         &       \\
0      & \ddots & \ddots    & \ddots\\
\end{array}
\right]
\]
of a QBD with space state $\mathbb{N} \times \s$, $\s=\{1,\ldots,n\}$,
where $A'_0$ and $A'_1$ are $n\times n$ matrices \cite{lr99}. We
assume that the matrix $P$ is irreducible and that the following
properties are satisfied, they are not restrictive for models of
practical interest:

\begin{assumption}
   \label{a:zero}
The matrix $A_{-1} +A_0 +A_1$ is irreducible.
\end{assumption}

\begin{assumption}
   \label{a:one}
The doubly infinite QBD on $\ZZ \times \s$ has only one final class
$\ZZ \times \s_*$, where $\s_* \subseteq \s$.  Every other state is on a
path to the final class.
Moreover, the set $\s_*$ is not empty.
\end{assumption}

Assumption \ref{a:one} is Condition 5.2 in \cite[Page 111]{blm:book} where it is implicitly assumed that 
$\s_*$ is not empty.

We denote by $G$, $R$, $\wh G$ and $\wh R$ the minimal nonnegative solutions of
the following equations
\begin{equation}\label{eq:4eq}
\begin{split}
&A_{-1}+(A_0-I)X+A_1X^2=0,\\
&X^2 A_{-1}+X(A_0-I)+A_1=0,\\
&A_{-1}X^2+(A_0-I)X+ A_{1}=0,\\
&A_{-1}+X (A_0-I)+ X^2 A_{1}=0,
\end{split}
\end{equation}
respectively.  The matrices $G$, $R$, $\wh G$ and $\wh R$ exist, are unique and have a probabilistic interpretation \cite{lr99}. If $S$ is any of $G$, $R$, $\wh G$ and $\wh R$, we denote by $\rho_S$ the spectral radius $\rho(S)$ of $S$, and we denote by $u_S$ and $v_S^T$ a nonnegative right and left Perron eigenvector of $S$, respectively, so that $Su_S=\rho_S u_S$ and $v_S^T S=\rho_S v_S^T$.

Define the matrix polynomial 
\[
B(z)=A_{-1}+z(A_0-I)+z^2A_1=B_{-1}+zB_0+z^2B_1,
\]
and the Laurent matrix polynomial $\varphi(z)=z^{-1}B(z)$.  Denote
by $\xi_1,\ldots,\xi_{2n}$ the roots of $B(z)$, ordered such  that
$|\xi_1|\le \cdots\le|\xi_{2n}|$.    According to Theorem~\ref{thm:w}
and \cite[Theorem 5.20]{blm:book}, the eigenvalues of $G$, $R$, $\wh
G$ and $\wh R$ are related as follows to the roots $\xi_i$, $i=1,\ldots,2n$.

\begin{theorem}\label{thm:spprop}
The eigenvalues of $G$ and $\wh R$ are $\xi_1,\ldots,\xi_n$, while the eigenvalues of $R$ and $\wh G$ are 
$\xi_{n+1}^{-1},\ldots,1/\xi_{2n}^{-1}$. Moreover $\xi_n$, $\xi_{n+1}$ are real positive, and:
\begin{enumerate}
\item if the QBD is positive recurrent then $\xi_n=1<\xi_{n+1}$, $G$ is stochastic and $\wh G$ is substochastic;
\item if the QBD is null recurrent then $\xi_n=1=\xi_{n+1}$,  $G$ and $\wh G$ are stochastic;
\item if the QBD is transient then $\xi_n<1=\xi_{n+1}$, $G$ is substochastic and $\wh G$ is stochastic.
\end{enumerate}
\end{theorem}

As a consequence of the above theorem we have $\rho_G=\rho_{\wh R}=\xi_n>0$, $\rho_R=\rho_{\wh G}=\xi_{n+1}^{-1}>0$. This way, since $A(z)=A_{-1}+zA_0+z^2A_1$ is irreducible and nonnegative for $z>0$, we find that
$u_G$, $u_{\wh G}$ are the positive Perron vectors of $A(\xi_n)$ and $A(\xi_{n+1}^{-1})$, respectively.  Similarly, $v_R$ and $v_{\wh R}$ are positive left Perron vectors of
 $A(\xi_{n+1}^{-1})$ and $A(\xi_n)$, respectively.
 Moreover, if $\xi_n=\xi_{n+1}$, then $u_G=u_{\wh G}=e$, , where $e$ is the vector of all ones, and $v_R=v_{\wh R}$ with $v_R^T(A_{-1}+A_0+A_1)=v_R^T$.
Under Assumption \ref{a:one}, according to
\cite[Section 4.7]{blm:book}, it follows that 1 is the only root of $B(z)$
of modulus 1.

Since $G$ and $R$ solve the first two equations in \eqref{eq:4eq}, we find that
\begin{equation}\label{lf}
\varphi(z)=(I-zR)K(I-z^{-1}G),~~K=A_0-I+A_1G=A_0-I+RA_{-1}.
\end{equation}
 Similarly, since $\wh G$ and $\wh R$ solve the last two equations in \eqref{eq:4eq}, we have
\begin{equation}\label{rf}
  \varphi(z^{-1})=(I-z \wh R)\wh K(I-z^{-1}\wh G),~~\wh K=A_0-I+A_{-1}\wh G=A_0-I+\wh RA_{1}.
\end{equation}

In view of Theorem \ref{thm:spprop} the decompositions \eqref{lf}  and
\eqref{rf} are weak canonical factorizations of $\varphi(z)$ and $\varphi(z^{-1})$,
respectively.
From \eqref{lf} and \eqref{rf} we have
\begin{equation}\label{eq:several}
\begin{split}
&A_1=-RK=-\wh K\wh G,\qquad A_{-1}=-KG=-\wh R\wh K,\\
& A_1G=RA_{-1},\qquad \qquad A_{-1}\wh G=\wh RA_1.
\end{split}
\end{equation}

The following result provides some properties of the matrices involved in the above equations.

\begin{theorem}\label{prop:4}
The following properties hold:
\begin{enumerate}
\item $-K$ and $-\wh K$ are nonsingular M-matrices;
\item $v_G^T K^{-1} u_R<0$ and $v_{\wh G}^T {\wh K}^{-1} u_{\wh R}<0$;
\item the series $\sum_{i=0}^\infty G^iK^{-1}R^i$ and $\sum_{i=0}^\infty \wh G^i{\wh K}^{-1}\wh R^i$ are convergent if and only if the QBD is not null recurrent.
\end{enumerate}
\end{theorem}

\begin{proof}
The matrix $U=A_0+A_1G$ is nonnegative and 
\[
Uu_G=(A_0+A_1G)u_G=\rho_G^{-1}(A_0G+A_1G^2)u_G.
\]
 Since $G$ solves the equation \eqref{eq:uqme},
then
$Uu_G=\rho_G^{-1}(G-A_{-1})u_G\le u_G$. Since $u_G>0$, this latter inequality implies that $\rho(U)\le 1$; moreover, $\rho(U)$ cannot be one otherwise $K$ would be singular, and from \eqref{lf} the polynomial $\det B(z)$ would be identically zero. Hence, $-K$ is a nonsingular M-matrix. Similarly, $-\wh K$ is a nonsingular M-matrix.
The proof of part 2 is rather technical and  is reported in the Appendix.
Concerning part 3, consider the series $\sum_{i=0}^\infty G^iK^{-1}R^i$. Since the matrix
$-K$ is a nonsingular M-matrix, one has $K^{-1}\le 0$, and
the series has nonpositive terms since $G\ge 0$ and $R\ge 0$. In the null recurrent case
$\rho(R)=\rho(G)=1$, therefore the series diverges since $v_G^T K^{-1} u_R<0$. In the other cases, the powers of $G$ and $R$ are uniformly bounded, and
one of the matrices $G$ and $R$ has spectral radius less than one, therefore the
series is convergent.
\end{proof}

In the non null recurrent case, the matrices $G$ and $\wh R$ on the
one hand, $\wh G$ and  $R$ on the other hand, are related through 
the series $W=\sum_{i=0}^\infty G^iK^{-1}R^i$ as indicated by  Part 3 of Theorem \ref{thm:w}.

\section{Shifting techniques for QBDs}\label{sec:shift}

The shift technique presented in this paper may be seen as an
extension, to matrix polynomials, of the following result  due to Brauer \cite{brauer}:

\begin{theorem}\label{thm:brauer}
Let $A$ be an $n\times n$ matrix with eigenvalues $\lambda_1,\ldots,\lambda_n$.
Let $x_k$ be an eigenvector of $A$ associated with the eigenvalue $\lambda_k$, $1\le k\le n$, and
let $q$ be any $n$-dimensional vector. The matrix 
$A + x_k q^T$ has eigenvalues
$\lambda_1,\ldots,\lambda_{k-1},\lambda_k+x_k^Tq,\lambda_{k+1},\ldots,\lambda_n$. 
\end{theorem}

The matrix polynomial $B(z)=A_{-1}+z(A_0-I)+z^2A_1$ has always a root on the unit circle, namely
$z=1$. This implies that $\varphi(z)=z^{-1}B(z)$ is not invertible on the unit circle and has only a weak canonical factorization (see formulas \eqref{lf} and \eqref{rf}). 
In this section we revisit in functional form the shift technique introduced in \cite{hmr01}.
Starting from $\varphi(z)$ 
we construct a new Laurent matrix polynomial
$\varphi_\xs(z)$ such that the roots of $B_\xs(z)=z\varphi_\xs(z)$ coincide
with the roots of $B(z)$ except for one root, which is shifted away to zero or to infinity. Therefore we may apply this technique to remove the singularities on the unit circle. 
This can be performed in two different ways: by operating to the right of $\varphi(z)$ or operating to the left. We treat separately the two cases.

\subsection{Shift to the right}
Our aim in this section is to shift the root $\xi_n$ of $B(z)$ to
zero. To this end, we multiply $\varphi(z)$ on the right by a suitable matrix function.

Take $Q=u_Gv^T$, where $v$ is any vector such that $u_G^Tv=1$.
Define
\begin{equation}\label{shift}
\varphi_\xr(z)=\varphi(z)\left(I+\frac {\xi_n}{z-\xi_n}Q\right). 
\end{equation}
where the suffix $\xr$ denotes shift to the Right.
We prove the following:

\begin{theorem}\label{thm:rs}
The function $\varphi_{\xr}(z)$ defined in \eqref{shift} coincides with the Laurent matrix polynomial $\varphi_{\xr}(z)=z^{-1}A^\xr_{-1}+A^\xr_0-I+z  A^\xr_1$ with matrix coefficients
\begin{equation}\label{eq:coeds0}
A^\xr_{-1}=A_{-1}(I-Q),~~ A^\xr_0=A_0+\xi_nA_1Q ,~~
 A^\xr_1=A_1.
\end{equation}
Moreover, the roots of $B_{\xr}(z)=z\varphi_\xr(z)$ are $0,\xi_1,\ldots,\xi_{n-1},\xi_{n+1},\ldots,\xi_{2n}$.
\end{theorem}

\begin{proof}
Since $\xi_n=\rho_G$ and $B(\xi_n)u_G=0$, then $A_{-1}Q=-\xi_n(A_0-I)Q-\xi_n^2A_1Q$, and  we have
\[
\begin{split}
B(z)Q
&=-\xi_n(A_0-I)Q-\xi_n^2A_1Q+(A_0-I)Qz+A_1Qz^2\\
&=(z^2-\xi_n^2)A_1Q+(z-\xi_n)(A_0-I)Q.
\end{split}
\]
 This way we find that 
$
\frac {\xi_n}{z-\xi_n}B(z)Q=\xi_n(z+\xi_n)A_1Q+\xi_n(A_0-I)Q$, therefore
\[
\varphi_{\xr}(z)=\varphi(z)+\frac {\xi_n}{z-\xi_n}\varphi(z)Q=z^{-1}A^\xr_{-1}+A^\xr_0-I+A^\xr_1z
\]
so that \eqref{eq:coeds0} follows.
As $\det (I+\frac {\xi_n}{z-\xi_n}Q)=\frac z{z-\xi_n}$,
 we have  from \eqref{shift} that $\det B_{\xr}(z)=\frac z{z-\xi_n}\det B(z) $. This means that the roots  of the polynomial 
$\det B_{\xr}(z)$ coincide with the roots of $\det B(z)$ except for $\xi_n$ which is replaced with 0.
\end{proof}

We analyze the consequences of the above theorem.
In the positive recurrent case, where $\xi_n=1<\xi_{n+1}$, 
 the matrix polynomial $B_\xr(z)$ has $n$ roots of modulus strictly less than 1, and $n$ of modulus strictly greater than 1; in particular, $B_{\xr}(z)$ is nonsingular on the unit circle and on the annulus $|\xi_{n-1}|<|z|<\xi_{n+1}$.
In the null recurrent case, where  $\xi_n=1=\xi_{n+1}$, the matrix polynomial $B_{\xr}(z)$ has $n$ roots of modulus strictly less than 1, and $n$ of modulus greater than or equal to 1; in particular, $B_{\xr}(z)$ has a simple root at $z=1$.
In the transient case, where $\xi_n<1=\xi_{n+1}$, the splitting of the
roots with respect to the unit circle is not changed, since $B_{\xr}(z)$ has, like $B(z)$, 
 $n$ roots of modulus strictly less than 1, and $n$ of modulus greater than or equal to 1.

It is worth pointing out that in the recurrent case the vector $u_G$ is the vector of all ones and $\xi_n=1$, therefore the quantities involved in the construction of the matrix polynomial $B_{\xr}(z)$ are known {\em a priori}. 
In the transient case it is convenient to apply the shift to the root $\xi_{n+1}$, by moving it away to infinity. This is obtained by acting on $\varphi(z)$ to the left, as described in the next section.

\subsection{Shift to the left}\label{sec:sr}
Consider the matrix $S=wv_R^T$, where $w$ is any vector such that  $v_R^Tw=1$.  
Define the matrix function
\begin{equation}\label{shiftl}
\varphi_{\xl}(z)=\left(I-\frac z{z-\xi_{n+1}}S\right)\varphi(z),
\end{equation}
where the suffix $\xl$ denotes shift to the left.

\begin{theorem}\label{thm:ls}
The function $\varphi_{\xl}(z)$ defined in \eqref{shiftl} coincides with the Laurent matrix polynomial $\varphi_{\xl}(z)=
z^{-1}A^\xl_{-1}+A^\xl_0-I+z A^\xl_1$ with matrix coefficients
$
A^\xl_{-1}=A_{-1}$,
$ A^\xl_0=A_0+\xi_{n+1}^{-1}SA_{-1}$,
 $A^\xl_1=(I-S)A_1$.
Moreover, the roots of $B_{\xl}(z)=z\varphi_\xl(z)$ are $\xi_1,\ldots,\xi_{n},\xi_{n+2},\ldots,\xi_{2n},\infty$.
\end{theorem}

\begin{proof}
Since $\xi_{n+1}=\rho_R^{-1}$ and $v_R^TB(\rho_R^{-1})=0$,
then $SA_{-1}=-\xi_{n+1}S(A_0-I)-\xi_{n+1}^{2}SA_1$, and we have
\[
\begin{split}
SB(z)&=
-\xi_{n+1}S(A_0-I)-\xi_{n+1}^{2}SA_1+S(A_0-I)z+SA_1z^2  \\
&=   
(z^2-\xi_{n+1}^{2})SA_1+(z-\xi_{n+1})S(A_0-I).
\end{split}
\]
This way we find that $\frac z{z-\xi_{n+1}}S\varphi(z)=(z+\xi_{n+1})SA_1+S(A_0-I)$, therefore
\[
\varphi_{\xl}(z)=\varphi(z)-\frac z{z-\xi_{n+1}}S\varphi(z)=z^{-1}A^\xl_{-1}+A^\xl_0-I+zA^\xl_1
\]
with
\[
A^\xl_{-1}=A_{-1},~~ A^\xl_0=A_0+\xi_{n+1}^{-1}SA_{-1} ,~~
 A^\xl_1=(I-S)A_1.
\]
As $\det (I-\frac z{z-\xi_{n+1}}S)=-\frac {1}{z-\xi_{n+1}}$,
we have $\det B_{\xl}(z)= -\frac
{1}{z-\xi_{n+1}}\det B(z) $ from~\eqref{shiftl}. This means that the roots of the
matrix polynomial $ B_{\xl}(z)$ coincide with the roots of $
B(z)$ except the root equal to $\xi_{n+1}$ which has been
 moved to infinity.
\end{proof}

A consequence of the above theorem is that in the transient case, when $\xi_n<1=\xi_{n+1}$, the matrix polynomial $B_{\xl}(z)$ has $n$ roots of modulus strictly less than 1 and $n$ roots of modulus strictly greater than 1 (included the root(s) at the infinity). In particular, $\varphi_{\xl}(z)$ is invertible on the unit circle and on the annulus $\xi_{n}<|z|<|\xi_{n+2}|$.

The shift to the left applied to the function $\varphi(z)$ in order to move the root $\xi_{n+1}$ to the infinity, can be viewed as a shift to the right applied to the function $\wt\varphi(z)=\varphi^T(z^{-1})$ to move the root $\xi_{n+1}^{-1}$ to zero. In fact, observe that the roots of $z\wt\varphi(z)$ are the reciprocals of the roots of $B(z)$ so that the roots $\xi_n$ and $\xi_{n+1}$ of $B(z)$
play the role of the roots $\xi_{n+1}^{-1}$ and $\xi_n^{-1}$ of $z\wt\varphi(z)$ respectively.
From \eqref{shift} we have
\[
 \wt\varphi_{\xr}(z)=\wt\varphi(z)\left(I+\frac{\xi_{n+1}^{-1}}{z-\xi_{n+1}^{-1}}Q'\right)
\]
for $Q'=v_Rw^T$. Taking the transpose in both sides of the above equation yields
\[
\wt\varphi^T_{\xr}(z)=\left(I+\frac{\xi_{n+1}^{-1}}{z-\xi_{n+1}^{-1}}{Q'}^T\right)\wt\varphi(z)^T
.\]
Replacing  $z$ with $z^{-1}$ yields \eqref{shiftl}
where $\varphi_{\xl}(z)=\wt\varphi^T_{\xr}(z^{-1})$.

\subsection{Double shift}
The right and left shifts presented in the previous sections can be combined,
yielding the double shift technique, where the new quadratic matrix polynomial
$B_{\xd}(z)$ has the same roots of $B(z)$, except for $\xi_n$ and
$\xi_{n+1}$, which are shifted to 0 and to infinity, respectively.

By following the same arguments used in the previous sections, we define the matrix function
\begin{equation}\label{eq:ds}
\varphi_{\xd}(z)= \left(I-\frac z{z-\xi_{n+1}}S\right)\varphi(z)\left(I+\frac {\xi_n}{z-\xi_n}Q\right),
\end{equation}
where $Q=u_Gv^T$ and $S=wv_R^T$, with
$v$ and $w$ any vectors such that $u_G^Tv=1$ and $v_R^Tw=1$.  
From Theorems \ref{thm:rs} and \ref{thm:ls} we find that   $\varphi_{\xd}(z)=z^{-1} A^{\xd}_{-1}+ A_0^{\xd}-I+z A^{\xd}_1$, with matrix coefficients
\begin{equation}\label{eq:coeds}
\begin{split}
A_{-1}^{\xd}& =A_{-1}(I-Q),\\
A_0^{\xd} &=A_0+\xi_nA_1Q+\xi_{n+1}^{-1}SA_{-1}- \xi_{n+1}^{-1}SA_{-1}Q\\
  &= A_0+\xi_nA_1Q+\xi_{n+1}^{-1}SA_{-1}- \xi_nSA_{1}Q\\
 A_1^{\xd} &=(I-S)A_1.
\end{split}
\end{equation}
The two expressions for $A_0^{\xd}$ coincide since, from \eqref{lf}, one has $A_1G=RA_{-1}$, and therefore $\xi_nv_R^TA_{1}u_G=\xi_{n+1}^{-1}v_R^TA_{-1}u_G$. 

From Theorems \ref{thm:rs} and \ref{thm:ls} it follows that the matrix polynomial $B_{\xd}(z)=z\varphi_{\xd}(z)$ has roots
$0,\xi_1,\ldots,\xi_{n-1},\xi_{n+2},\ldots,\xi_{2n},\infty$. In
  particular, $\varphi_{\xd}(z)$ is nonsingular on the unit circle and on the annulus 
  $|\xi_{n-1}|<|z|<|\xi_{n+2}|$.

\section{Canonical factorizations}\label{sec:cf}
Consider the Laurent matrix polynomial $\varphi_{\xs}(z)$, for $s\in\{\xr,\xl,\xd\}$, 
where $\varphi_{\xs}(z)$ is obtained by applying one of the shift techniques
described in Section \ref{sec:shift}.  Our  goal in this section is to
show that $\varphi_{\xs}(z)$ and $\varphi_{\xs}(z^{-1})$ admit a (weak)
canonical factorization, and to determine relations between $G$, $R$, $\wh
G$ and $\wh R$, and the solutions
of the transformed equations

\begin{eqnarray}
&A^\xs_{-1}+(A^\xs_0-I)X+A^\xs_1X^2=0,\label{eq:gs}\\
&X^2 A^\xs_{-1}+X(A^\xs_0-I)+A^\xs_1=0,\label{eq:rs}\\
&A^\xs_{-1}X^2+(A^\xs_0-I)X+ A^\xs_{1}=0,\label{eq:ghs}\\
&A^\xs_{-1}+X (A^\xs_0-I)+ X^2 A^\xs_{1}=0.\label{eq:rhs}
\end{eqnarray}

\subsection{Shift to the right}
Consider the function $\varphi_{\xr}(z)$ obtained by shifting $\xi_n$ to
zero, defined in \eqref{shift}.  Independently of the
recurrent/transient case, the matrix Laurent polynomial $
\varphi_{\xr}(z)$ has a canonical factorization, as shown by
the following theorem.

\begin{theorem}\label{thm:frs}
Define  $Q=u_Gv^T$, where $v$ is any vector such that $u_G^Tv=1$.
The 
function $\varphi_{\xr}(z)$, defined in \eqref{shift}, has the
factorization
\[
\varphi_{\xr}(z) =(I-zR_\xr)K_\xr(I-z^{-1}G_{\xr}),
\]
where $G_{\xr}=G-\xi_n Q$,  $R_\xr=R$ and $K_\xr=K$. 
This factorization is canonical in the positive recurrent case, and weakly canonical otherwise.
Moreover, the eigenvalues
of $G_{\xr}$ are those of $G$, except for the eigenvalue $\xi_n$ which
is replaced by zero; the matrices $G_{\xr}$ and $R_\xr$ are the solutions with minimal spectral radius of the equations 
\eqref{eq:gs} and \eqref{eq:rs}, respectively.
\end{theorem}

\begin{proof}
Since $GQ=\xi_nQ$, then $(I-z^{-1}G)(I+\frac {\xi_n}{z-\xi_n}Q)=I-z^{-1}(G-\xi_n Q)$.
Hence, from \eqref{lf} and \eqref{shift}, we find that
\[
\varphi_{\xr}(z)=(I-zR)K(I-z^{-1}G_{\xr}),~G_{\xr}=G-\xi_n Q,
\]
which proves the factorization of $\varphi_\xr(z)$.
Since $\det (I-z^{-1}G_{\xr})=\frac z{z-\xi_n}\det (I-z^{-1}G)$, then
the eigenvalues of $G_{\xr}$ are the eigenvalues of $G$, except for the
eigenvalue $\xi_n$ which is replaced by 0. Thus the factorization is
canonical in the positive recurrent case, weak canonical otherwise.
A direct inspection shows that $G_{\xr}$ and $R$ solve 
\eqref{eq:gs} and \eqref{eq:rs}, respectively.
 They are the solutions with minimal spectral radius  since their eigenvalues coincide with the $n$ roots with smallest modulus of $B_\xr(z)$ and of $zB_\xr(z^{-1})$, respectively.
\end{proof}

For the existence of the (weak) canonical factorization of $\varphi_{\xr}(z^{-1})$ we distinguish the null recurrent from the non null recurrent case.
In the latter case, since the matrix polynomial $B_{\xr}(z)$ is still singular on the unit circle, the function 
$\varphi_{\xr}(z^{-1})$  has a weak canonical factorization, as stated by the following theorem.

\begin{theorem}\label{thm:wcfnr}
Assume that $\xi_n=\xi_{n+1}=1$ (i.e., the QBD is null recurrent); define $Q=u_Gv_{\wh G}^T$,  where $v_{\wh G}^Tu_G=1$. Normalize $u_{\wh R}$ so that
 $v_{\wh G}^T\wh K^{-1}u_{\wh R}=-1$.   The function 
$\varphi_{\xr}(z)$,  defined in
\eqref{shift}, has the weak canonical factorization
\[
\varphi_{\xr}(z^{-1}) =(I-z\wh R_{\xr})\wh K_\xr(I-z^{-1}\wh G_{\xr})
\]
with 
\begin{eqnarray}
\wh R_{\xr}=\wh R+ u_{\wh R} v_{\wh G}^T \wh K^{-1},\label{eq:rht}\\
\wh K_{\xr}= \wh K-(u_{\wh R}+\wh K u_G)v_{\wh G}^T,\label{eq:kht}\\
\wh G_{\xr}=\wh G+(u_G+\wh K^{-1}u_{\wh R})v_{\wh G}^T. \label{eq:ght}
\end{eqnarray}
The eigenvalues of $\wh R_{\xr}$ are those of $\wh R$, except for the eigenvalue 1 which is replaced by 0; the eigenvalues of $\wh G_{\xr}$ are the same as the eigenvalues of $\wh G$. Moreover, the matrices $\wh G_{\xr}$ and $\wh R_{\xr}$ are the solutions of minimum spectral radius of \eqref{eq:ghs} and \eqref{eq:rhs}, respectively. 
\end{theorem}
\begin{proof}
Since both $G$ and $\wh G$ are stochastic, $u_G = u_{\wh G} = e$.  As
 $u_G$ and $v_{\wh G}$ are right and left eigenvector,
respectively, of $\wh G$ corresponding to the same eigenvalue, then
$v_{\wh G}^Tu_G\ne 0$ and we may scale the vectors in such a way that
$v_{\wh G}^Tu_G=1$. In view of part 2 of Theorem \ref{prop:4} we have $v_{\wh G}^T\wh K^{-1}u_{\wh R}<0$, so that we may normalize $u_{\wh R}$ so that
$v_{\wh G}^T\wh K^{-1}u_{\wh R}=-1$.
Observe that, for the matrix $\wh R_{\xr}$ of \eqref{eq:rht}, we have  
$\wh R_{\xr}u_{\wh R}=u_{\wh R}+u_{\wh R}( v_{\wh G}^T\wh K^{-1} u_{\wh R})=0$.
From this property, in view of Theorem \ref{thm:brauer}, it follows that the eigenvalues of $\wh R_{\xr}$ are those of 
${\wh R}$, except for the eigenvalue 1, which is replaced by 0. Similarly, for the matrix
 $\wh G_{\xr}$ of \eqref{eq:ght}, one finds that 
 $v_{\wh G}^T \wh G_{\xr}=v_{\wh G}^T\wh  G= v_{\wh G}^T$, therefore
 the matrix $\wh G_{\xr}$ has the same eigenvalues of $\wh G$ for
 Theorem \ref{thm:brauer}.
 Now we prove that $\wh R_{\xr}$ solves equation \eqref{eq:rhs}. 
 By replacing $X$ with $\wh R_{\xr}$ and the block coefficients with the expressions in \eqref{eq:coeds0}, the left hand side of
equation \eqref{eq:rhs}  becomes
$A_{-1}(I-Q)+\wh R_{\xr}(A_0-I+A_1Q)+\wh R_{\xr}^2A_1.$
Observe that 
$\wh R_{\xr}^2={\wh R}^2+ u_{\wh R}v_{\wh G}^T \wh K^{-1}\wh R$.
By replacing $\wh R_{\xr}$ and $\wh R_{\xr}^2$ with their expressions in terms of $\wh R$, and by using the property
$ A_{-1}+\wh RA_0+\wh R^2A_1=\wh R$, 
the left hand side of equation \eqref{eq:rhs}  becomes
\[
\begin{split}
A_{-1} (I-Q)&+\wh R_{\xr}(A_0-I+A_1Q)+\wh R_{\xr}^2A_1  \\ 
&= -A_{-1}Q+ \wh R A_1Q +  u_{\wh R}v_{\wh G}^T  \wh K^{-1} (A_0-I+A_1Q) 
+u_{\wh R}v_{\wh G}^T\wh K^{-1}\wh RA_1
 \\
&= -A_{-1}Q +A_{-1}\wh GQ - u_{\wh R}v_{\wh G}^T\wh GQ+u_{\wh R}v_{\wh G}^T  \wh K^{-1}
(A_0-I+\wh R A_1) \\
&=-u_{\wh R}v_{\wh G}^TQ+u_{\wh R}v_{\wh G}^T=0,
\end{split}
\]
where the first equality holds since $\wh RA_1=A_{-1}\wh G$ and $\wh K^{-1}A_1=-\wh G$, the second and third equalities hold since $\wh GQ=Q$, $\wh K=A_0-I+\wh RA_1$ and $v_{\wh G}^TQ=v_{\wh G}^T$.
In view of Theorem \ref{thm:w}, where the role of $G$ is replaced by $R$, the function $\varphi_{\xr}(z^{-1})$ 
has the desired weak canonical factorization
where
$\wh K_{\xr}=A^\xr_0-I+\wh R_{\xr} A^\xr_1$ and $\wh G_{\xr}=-\wh
  K_{\xr}^{-1}A^\xr_1$. 
To prove that $\wh K_{\xr}$ is given by \eqref{eq:kht}, we replace the
expression \eqref{eq:rht} of $\wh R_{\xr}$ in $\wh K_{\xr}$ and this 
yields
\[
\begin{split}
\wh K_{\xr}&=A_0-I+A_1Q+(\wh R+u_{\wh R}v_{\wh G}^T\wh K^{-1})A_1\\
&= A_0-I+\wh RA_1+u_{\wh R}v_{\wh G}^T\wh K^{-1}A_1+A_1Q\\
&=\wh K+(-u_{\wh R}+A_1u_G)v_{\wh G}^T
=\wh K-(u_{\wh R}+\wh K u_G)v_{\wh G}^T.
\end{split}
\]
Here we have used the properties $\wh K^{-1}A_1=-\wh G$, $v_{\wh G}^T\wh G= v_{\wh G}^T$, $\wh Gu_{G}=u_G$, and $\wh K=A_0-I+\wh RA_1$.
Finally,  we prove that $\wh G_{\xr}$ is given by \eqref{eq:ght}.
By using the Sherman-Woodbury-Morrison formula we may write
 $\wh K_{\xr}^{-1}=\wh K^{-1}+\gamma \wh K^{-1}(u_{\wh R}+\wh K u_G)v_{\wh G}^T\wh K^{-1}$, where $\gamma=1/(1-v_{\wh G}^T\wh K ^{-1}(u_{\wh R}+\wh K u_G))=1$ for the assumptions on $v_{\wh G}$, $u_{\wh R}$ and $u_G$. Hence,  
  $\wh K_{\xr}^{-1}=\wh K^{-1}+(\wh K^{-1} u_{\wh R}+ u_G)v_{\wh G}^T\wh K^{-1}$ 
  so that
 $\wh G_{\xr}=-\wh K_{\xr}^{-1}A^\xr_1= \wh G+(u_G+\wh K^{-1}u_{\wh R})v_{\wh G}^T$.
\end{proof}

In the non null recurrent case, the function  $\varphi_{\xr}(z^{-1})$  has a (weak) canonical factorization, as stated by the following theorem.

\begin{theorem}\label{thm:crspr}
Assume that $\xi_n<\xi_{n+1}$ (i.e., the QBD is not null recurrent).
Define  $Q=u_Gv^T$,  with $v$ any vector such that $u_G^Tv=1$ and $\xi_nv^T \wh G u_G\ne 1$. The Laurent matrix polynomial 
$\varphi_{\xr}(z^{-1})$ defined in \eqref{shift},
has the  factorization
\[
\varphi_{\xr}(z^{-1}) =(I-z \wh R_{\xr})\wh K_{\xr}(I-z^{-1}\wh G_{\xr}),
\]
  where 
\[\begin{split}
W_{\xr}=&W-\xi_n QWR, \\
\wh K_{\xr}=& A^\xr_0-I+A^\xr_{-1}\wh G_{\xr}= A^\xr_0-I+\wh R_{\xr} A_{1},\\
G_{\xr}=&G-\xi_nQ,\\
\wh G_{\xr}=&W_{\xr} R W_{\xr}^{-1},\\
\wh R_{\xr}=&W_{\xr}^{-1}G_{\xr} W_{\xr}.
\end{split}
\] 
Moreover, $\wh G_{\xr}$ and $\wh R_{\xr}$ are the solutions with minimal
spectral radius of \eqref{eq:ghs} and \eqref{eq:rhs}, respectively. 
  The factorization
is canonical if
$\xi_n=1$ and weakly canonical if $\xi_{n+1}=1$.
\end{theorem}

\begin{proof}
As a first step, we show that the matrix $W_{\xr}=\sum_{i=0}^{+\infty}G_{\xr}^iK^{-1}R^i$, with $G_{\xr}=G-\xi_n Q$, is nonsingular, so that we can apply property 3~of Theorem \ref{thm:h0} to the matrix Laurent polynomial   $\varphi_{\xr}(z)$ of Theorem \ref{thm:rs}.
Observe that $G_{\xr}^i=G^i-\xi_n QG^{i-1}$, for $i\ge 1$. Therefore, we may write
\[
\begin{split}
W_{\xr}&=K^{-1}+\sum_{i=1}^{+\infty}(G^i-\xi_n QG^{i-1})K^{-1}R^i\\
&=K^{-1}+\sum_{i=1}^{+\infty}G^iK^{-1}R^i-\xi_n Q\left(\sum_{i=0}^{+\infty}G^iK^{-1}R^i \right) R
=W-\xi_n QWR.
\end{split}
\]
Since $\det W\ne 0$ by Theorem \ref{thm:w}, part 3, then $\det 
W_{\xr}=\det(I-\xi_n QWRW^{-1})\det W$. Moreover, since $Q=u_G v^T$, then
the matrix $I-\xi_n QWRW^{-1}$ is nonsingular if and only if $\xi_n
v^T WRW^{-1} u_G\ne 1$.   Since $\wh G=WRW^{-1}$, the latter condition
holds if $\xi_nv^T \wh G u_G\ne 1$, which we assume, and so,
 the matrix $W_{\xr}$ is nonsingular.
If $\xi_n=1$, since $\rho(G_{\xr})<1$ and $\rho(R)<1$, from 3 of Theorem \ref{thm:h0} applied to  the matrix Laurent polynomial   $\varphi_{\xr}(z)$, we deduce that $\varphi_{\xr}(z^{-1})$ has the  canonical factorization 
$\varphi_{\xr}(z^{-1}) =(I-z\wh R_{\xr})\wh K_{\xr}(I-z^{-1}\wh G_{\xr})$ with
$\wh K_{\xr}=A_0^\xr-I+A^\xr_{-1}\wh G_{\xr}=A^\xr_0-I+\wh R_{\xr} A_{1}$ and $\wh G_{\xr}=W_{\xr}
 R  W_{\xr}^{-1}$,  $\wh R_{\xr}= W_{\xr}^{-1}  G_{\xr} W_{\xr}$. If $\xi_{n+1}=1$, we can apply the above property to the function $\varphi^{(t)}(z)=\varphi(tz)$ with $\xi_{n}<t<1$ and obtain the canonical factorization for $\varphi^{(t)}(z)$.
Scaling again the variable $z$ by $t^{-1}$ we obtain a weak canonical factorization for $\varphi(z)$. 
With the same arguments 
used in the proof of Theorem \ref{thm:frs}, we may prove that
$\wh G_{\xr}$ and $\wh R_{\xr}$ are the solutions with minimal spectral radius of  \eqref{eq:ghs} and \eqref{eq:rhs}.
\end{proof}

In the above theorem we can choose $v=v_{\wh G}$, so that $v_{\wh G}^T\wh G=\xi_{n+1}^{-1} v_{\wh G}^T$. Since $u_G>0$,  then  $v_{\wh G}^Tu_G>0$ and we can normalize the vectors so that $v_{\wh G}^Tu_G=1$. In this way we obtain 
$
\xi_nv_{\wh G}^T \wh Gu_G=\xi_n \xi_{n+1}^{-1} v_{\wh G}^Tu_G=\xi_n\xi_{n+1}^{-1}<1.
$
 Therefore, the assumption on $v$ of Theorem \ref{thm:crspr} is satisfied.

\subsection{Shift to the left}
As for the right shift,  the matrix Laurent polynomial $
\varphi_{\xl}(z)$  defined
by \eqref{shiftl} and obtained by shifting $\xi_{n+1}$ to infinity, has a canonical factorization, as shown by the following theorem.

\begin{theorem}\label{thm:fls}
Define  $S=wv_R^T$, where $w$ is any vector such that $v_R^Tw=1$.
The  function $\varphi_{\xl}(z)$ defined in \eqref{shiftl}, has the  factorization
\[
\varphi_{\xl}(z) =(I-zR_{\xl})K_\xl(I-z^{-1} G_\xl),
\]
where $R_{\xl}=R-\xi_{n+1}^{-1} S$, $G_\xl=G$ and $K_\xl=K$. 
This factorization is canonical in the transient case, weakly canonical otherwise.
Moreover, the eigenvalues
of $R_{\xl}$ are those of $R$, except for the eigenvalue $\xi_{n+1}^{-1}$ which
is replaced by zero; the matrices $G$ and $R_{\xl}$ are the solutions with minimal spectral radius of equations \eqref{eq:gs} and \eqref{eq:rs}, respectively.
\end{theorem}

\begin{proof}
The proof can be carried out as the proof of Theorem \ref{thm:frs}, after observing that $(I-\frac {z}{z-\xi_{n+1}}S)(I-zR)=I-z(R-\xi_{n+1}^{-1} S)$.
\end{proof}

Similarly to the shift to the right, we may prove the following results concerning the canonical factorization of $\varphi_{\xl}(z^{-1})$:

\begin{theorem}\label{thm:wcfnnrl}
Assume that $\xi_n=\xi_{n+1}=1$ (i.e., the QBD is null recurrent).
Define $S=u_{\wh R} v_{R}^T$,  where $v_{R}^Tu_{\wh
  R}=1$. Normalize $v_{\wh G}$ such that $v_{\wh G}^T\wh K^{-1}u_{\wh
  R}=-1$.   The function
$\varphi_{\xl}(z)$ defined in \eqref{shiftl} has the weak canonical factorization
\[
\varphi_{\xl}(z^{-1}) =(I-z\wh R_{\xl})\wh K_{\xl}(I-z^{-1}\wh G_{\xl})
\]
with 
\[\begin{split}
\wh R_{\xl}&=\wh R+ u_{\wh R} (v_{R}^T +v_{\wh G}^T\wh K^{-1}), \\
\wh K_{\xl}&= \wh K-u_{\wh R}(v_{\wh G}^T+v_R^T\wh K),\\
\wh G_{\xl}&=\wh G+\wh K^{-1}u_{\wh R}v_{\wh G}^T.
\end{split}
\]
The eigenvalues of $\wh G_{\xl}$ are those of $\wh G$, except for the eigenvalue 1 which is replaced by 0; the eigenvalues of $\wh R_{\xl}$ are the same as the eigenvalues of $\wh R$. Moreover, the matrices $\wh G_{\xl}$ and $\wh R_{\xl}$ are the solutions of minimum spectral radius of  equations \eqref{eq:ghs} and \eqref{eq:rhs}, 
 respectively.  
\end{theorem}

If $\xi_n<\xi_{n+1}$ we have the following result.

\begin{theorem}\label{thm:wcfnr2}
Assume that $\xi_n<\xi_{n+1}$ (i.e., the QBD is not null recurrent).
Define  $Q=w v_R^T$,  with $w$ any vector such that $v_R^Tw=1$ and $\xi_{n+1}^{-1}v_R^T \wh R w\ne 1$. The Laurent matrix polynomial 
$\varphi_{\xl}(z^{-1})$
defined in \eqref{shiftl}, 
has the  factorization
\[
\varphi_{\xl}(z^{-1}) =(I-z\wh R_{\xl})\wh K_{\xl}(I-z^{-1}\wh G_{\xl})
\]
with
\[
\begin{split}
\wh G_{\xl}&=W_{\xl}^{-1}R_{\xl} W_{\xl},\\
W_{\xl}&=W-\xi_{n+1}^{-1} GWS,\\
R_{\xl}&=R-\xi_{n+1}^{-1}S,\\
\wh R_{\xl}&= W_{\xl} G W_{\xl}^{-1},\\
\wh K_{\xl}&= A^\xl_0-I+A^\xl_{-1}\wh G_{\xl}=A^\xl_0-I+\wh R_{\xl} A_{1}.
\end{split}
\]
Moreover, $\wh G_{\xl}$ and $\wh R_{\xl}$ are the solutions with minimal spectral radius of equations
\eqref{eq:ghs} and  \eqref{eq:rhs},  respectively.
 The factorization
is canonical if $\xi_n<1$, is weakly canonical if $\xi_{n}=1$.
\end{theorem}

\subsection{Double shift}
Consider the matrix function $\varphi_{\xd}(z)$ defined in \eqref{eq:ds}, obtained by shifting $\xi_n$ to 0 and $\xi_{n+1}$ to $\infty$.
The matrix Laurent polynomial $\varphi_{\xd}(z)$, has a canonical factorization, as shown by the following theorem.

\begin{theorem}\label{thm:fds}
Define  $Q=u_Gv^T$ and $S=wv_R^T$, with
$v$ and $w$ any vectors such that $u_G^Tv=1$ and $v_R^Tw=1$.
The  function $\varphi_{\xd}(z)$ defined in \eqref{eq:ds}, has the following canonical factorization
\[
\varphi_{\xd}(z) =(I-zR_{\xd})K_\xd(I-z^{-1} G_{\xd}),
\]
where $R_{\xd}=R-\xi_{n+1}^{-1} S$, $G_{\xd}=G-\xi_n Q$ and $K_\xd=K$.
Moreover,    $G_{\xd}$ and $R_{\xd}$ are the solutions with minimal spectral radius of  equations \eqref{eq:gs} and \eqref{eq:rs}, respectively.
 \end{theorem}

\begin{proof}
The proof can be carried out as the proof of Theorems \ref{thm:frs} and \ref{thm:fls}, since $(I-z^{-1}G)(I+\frac {\xi_n}{z-\xi_n}Q)=I-z^{-1}(G-\xi_n Q)$ and $(I-\frac {z}{z-\xi_{n+1}}S)(I-zR)=I-z(R-\xi_{n+1}^{-1} S)$.
\end{proof}

We show that in the null recurrent case, where $\xi_n=\xi_{n+1}=1$,
the matrix Laurent polynomial $\varphi_{\xd}(z^{-1})$
has also a canonical factorization:

\begin{theorem}\label{thm:lfnr}
Assume that $\xi_n=\xi_{n+1}=1$.  Define  $Q=u_Gv_{\hat G}^T$ and $S=u_{\hat R}v_R^T$, with
 $u_G^Tv_{\hat G}=1$ and $v_R^Tu_{\hat R}=1$.
Normalize the vectors $v_{\wh G}$ and $u_{\wh R}$ so that $v_{\wh G}^T\wh K^{-1}u_{\wh R}=-1$. 
The  function $\varphi_{\xd}(z^{-1})$  defined in \eqref{eq:ds}, has the following canonical factorization
\[
\varphi_{\xd}(z^{-1}) =(I-z\wh R_{\xd})\wh K_{\xd}(I-z^{-1} \wh G_{\xd})
\]
where
\[
\begin{split}
 \wh R_{\xd}&=\wh R+ u_{\wh R}v_{\wh G}^T\wh K^{-1},\\
 \wh G_{\xd}&=\wh G+\wh K^{-1} u_{\wh R}v_{\wh G}^T,\\ 
\wh K&=A_{-1}\wh G+A_0-I,\\
\wh K_{\xd}&=\wh K- u_{\wh R}v_{\wh G}^T.
\end{split}
\] 
Moreover,  the matrices $\wh G_{\xd}$ and $\wh R_{\xd}$ are the solutions with minimal spectral radius of  equations 
\eqref{eq:ghs} and \eqref{eq:rhs},
 respectively.
 \end{theorem}

\begin{proof}
In view of part 2 of Theorem \ref{prop:4}, we have $v_{\wh G}^T\wh K^{-1}u_{\wh R}<0$, therefore we may normalize the vectors so that $v_{\wh G}^T\wh K^{-1}u_{\wh R}=-1$.
Observe that, for the matrix $\wh G_{\xd}$ defined in the theorem, we have 
\[
Q\wh G_{\xd}=u_G v_{\wh G}^T (\wh G+\wh K^{-1} u_{\wh R}v_{\wh G}^T)=
u_G( v_{\wh G}^T+ (v_{\wh G}^T \wh K^{-1} u_{\wh R}) v_{\wh G}^T)=0.
\] 
Similarly, one has  $\wh R_{\xd}S=0$. From Theorem \ref{thm:brauer}, it follows that the eigenvalues of $\wh G_{\xd}$ are those of ${\wh G}$, except for the eigenvalue 1, which is replaced by 0; the same holds for $\wh R_{\xd}$. Now we prove that $\wh G_{\xd}$ solves the equation \eqref{eq:ghs}.
By replacing $X$ with $\wh G_{\xd}$ and the block coefficients with the expressions in \eqref{eq:coeds}, the left hand side of the quadratic equation \eqref{eq:ghs} becomes
\[
A_{-1}(I-Q)\wh G_{\xd}^2+(A_0-I+(I-S)A_1Q+SA_{-1})\wh G_{\xd}+(I-S)A_1.
\]
Since $Q\wh G_{\xd}=0$, the above expression 
simplifies to
\[
A_{-1}\wh G_{\xd}^2+(A_0-I+SA_{-1})\wh G_{\xd}+(I-S)A_1.
\]
Observe that 
$\wh G_{\xd}^2={\wh G}^2+{\wh G}\wh K^{-1} u_{\wh R}v_{\wh G}^T$.
By replacing $\wh G_{\xd}$ and $\wh G_{\xd}^2$ with their expressions in terms of $\wh G$, and by using the property
$ A_{-1}\wh G^2+(A_0-I)\wh G+A_1=0$, we get
\[
\begin{split}
&A_{-1}\wh G_{\xd}^2+(A_0-I+SA_{-1})\wh G_{\xd}+(I-S)A_1=\\
& A_{-1}\wh G \wh K^{-1} u_{\wh R}v_{\wh G}^T + (A_0-I) \wh K^{-1} u_{\wh R}v_{\wh G}^T +SA_{-1}\wh G + S A_{-1} \wh K^{-1} u_{\wh R}v_{\wh G}^T -SA_1=\\
&(A_{-1}\wh G +A_0-I)\wh K^{-1} u_{\wh R}v_{\wh G}^T +SA_{-1}\wh G + S A_{-1} \wh K^{-1} u_{\wh R}v_{\wh G}^T -SA_1+\wh K^{-1} u_{\wh R}v_{\wh G}^T=\\
& u_{\wh R}v_{\wh G}^T +SA_{-1}\wh G + S A_{-1} \wh K^{-1} u_{\wh R}v_{\wh G}^T -SA_1.
\end{split}
\]
This latter equation is zero. Indeed, $-SA_{-1} \wh K^{-1}=S\wh R=S$ and $Su_{\wh R}=u_{\wh R}$, therefore $S A_{-1} \wh K^{-1} u_{\wh R}v_{\wh G}^T= -u_{\wh R}v_{\wh G}^T$; moreover, $SA_{-1}\wh G=S\wh RA_1=SA_1$.
Similarly, we may prove that $\wh R_{\xd}$ solves the 
equation $A^\xd_{-1}+XA^\xd_0-I)+X^2A^\xd_1=0$.
Since $\wh G_{\xd}$ and $\wh R_{\xd}$ are the solutions of minimal spectral radius of  equations
\eqref{eq:ghs} and \eqref{eq:rhs}, we may apply
Theorem 3.20 of \cite{blm:book} and conclude that $\varphi_{\xd}(z^{-1})$ has desired the canonical factorization 
with 
\[\begin{split}
\wh K_{\xd}=&A^\xd_{-1}\wh G_{\xd}+
A^\xd_0-I=
A^\xd_{-1}({\wh G}+\wh K^{-1}u_{\wh R}v_{\wh G}^T)+A^\xd_0-I\\
=&
\wh K+ A^\xd_{-1}\wh K^{-1}u_{\wh R}v_{\wh G}^T=\wh K- u_{\wh R}v_{\wh G}^T,
\end{split}\] since $A^\xd_{-1}\wh K^{-1}=-\wh R$.
\end{proof}

\section*{Appendix}

Here we provide the proof of part 2 of Theorem \ref{prop:4}, i.e., $v_G^T K^{-1} u_R<0$ and $v_{\wh G}^T \wh K^{-1} u_{\wh R}<0$.
 The proof is based on an argument
of accessibility for the states of the doubly infinite  QBD.
Assumptions \ref{a:zero} and \ref{a:one},   together imply that $\s_* = \s$ and so we have
the following property.\smallskip

\noindent 
{\bf A.1}  For any $i,j\in\s$, for any level
 $k$ and $k'$, there is a path from $(k,i)$ to $(k',j)$.

We prove that $v_G^T K^{-1} u_R<0$; the proof that $v_{\wh G}^T \wh K^{-1} u_{\wh R}<0$ is similar and is left to the reader.
For the sake of simplicity we write $u$ in place of $u_R$ and $v$ in place of $v_G$.
The proof consists in analyzing the sign properties of the components
of the vectors $u$ and $-K^{-1}u$, and relies on the irreducibility assumptions. 

\paragraph{The vector $u$}
\label{sec:vector-u}
\  \\
{\it Case 1.}  $R$ is irreducible.  Then $u > 0$, and $K^{-1} u < 0$ also, so that $v^T K^{-1} u  < 0$.

\noindent
{\it Case 2.}  $R$ is reducible.  
\quad
We need to define various subsets of the set of phases $\s = \{1, \ldots, n\}$.  The most
important ones define  a partition of $\s$ into the four subsets
$\s_1$, $\wt\s_1$, $\wt\s_b$ and $\s_a$ that we define later.
As we assume that the matrix $A$ is irreducible, we have after a
suitable permutation of rows and columns
$\left[\begin{smallmatrix}R_1 & R_{12} \\ 0 & R_2\end{smallmatrix}\right]$,
where $R_1$ is irreducible and $R_2$ is strictly upper-triangular.
The proof is in \cite[Theorem 7.2.2, page 154]{lr99}. 
The rows and columns of $R_1$ are indexed by $\s_1$ and those of $R_2$
are indexed by $\s_2$, and so
$
\s = \s_1 \cup \s_2.
$  
By
Assumption \ref{a:one}, $\s_1$ is not empty.

Concerning the eigenvector $u$ of $R$, we have : $u_i > 0$ for any $i$
in $\s_1$ and $u_i=0$ for $i$ in $\s_2$.  The physical meaning of the
partition $\s = \s_1 \cup \s_2$ is given below.  Consider the doubly
infinite QBD process.
\smallskip 

\noindent
{\bf B.1.} For any $i\in\s_1$, for any level $k$, for any displacement
$h \geq 1$, there exists $i'\in\s_1$ such that there is a path from
$(k,i)$ to $(k+h, i')$ avoiding level $k$ and the levels below.
\smallskip

\noindent
{\bf B.2.} For any $i\in\s_2$, for any level $k$, for any displacement
$h \geq 1$, for any $i'\in\s_1$, there is no path from $(k,i)$ to
$(k+h,i')$ avoiding level $k$;  i.e.,
any path from $(k,i)$ to $(k+h,i')$ has to go through level $k$ or $k-1$.

\paragraph{The vector $K^{-1}u$}
\label{sec:vector--k}
\ 

\noindent
The matrix $-K^{-1}$ is about transitions within a level $k$ without
visiting level $k-1$: $(-K^{-1})_{ij}$ is the expected number of
visits to $(k,j)$, starting from $(k,i)$ before any visit to level
$k-1$.  Clearly, 
\smallskip

\noindent
{\bf C.1.} $(-K^{-1})_{ij}>0$ if and only if there exists a path from
$(k,i)$ to $(k,j)$ that avoids level $k-1$, possibly after visiting some
states in level $k+1$ or above.
\smallskip

Define $w = -K^{-1} u$.  We have $w \geq u$, so that $w_i >0$ for all
$i$ in $\s_1$.  Furthermore, there may be some phases $i$ in $\s_2$
such that $w_i>0$; define $\wt\s_1$ as the subset of phases $i$ such
that $u_i=0$, $w_i>0$, and define $\wt\s_2 = \s_2 \setminus \wt\s_1$,
thus
$
\s_2 = \wt\s_1 \cup \wt\s_2
$. \smallskip

\noindent
{\em Case 1.} $w_i > 0$ for all $i$, i.e., $\wt\s_2$ is
empty.  Then, $v^T w >0$ and we are done.\smallskip

\noindent
{\em Case 2.} The set $\wt\s_2$ is not empty, there are phases
$i$ such that $w_i =0$.  From B.1, B.2 and C.1, we find that the
physical meaning is as follows.\smallskip

\noindent
{\bf D.1.} 
For any $i\in\wt\s_1$, for any level $k$, for any displacement $h \geq 1$, there exists $i'\in\s_1$ such that there is a path from $(k,i)$ to $(k+h, i')$ avoiding level $k-1$ and any level below.
\smallskip

\noindent
{\bf D.2.}
For any phase $i\in\wt\s_2$, for any level $k$, for any
displacement $h \geq 0$, for any $i'\in\s_1$, there is no path from
$(k,i)$ to $(k+h,i')$ avoiding level $k-1$.

\paragraph{The product $v^T K^{-1} u$}
\label{sec:product--vt}
\ 

\noindent
By \cite[Theorem 7.2.1, page 152]{lr99}, we have after a suitable
permutation of rows and columns
$
G =\left[\begin{smallmatrix}G_a & 0 \\  G_{ba} & G_b\end{smallmatrix}\right]
$
where $G_a$ is irreducible and $G_b$ is strictly lower triangular.
The rows and columns of $G_a$ are indexed by $\s_a$ and those of $G_b$
are indexed by $\s_b$; define $n_b=|\s_b|$,
thus,
$
\s = \s_a \cup \s_b.
$  
The left eigenvector $v$ of $G$ is such that $v_i > 0$ for $i$ in
$\s_a$ and $v_i=0$ for $i$ in~$\s_b$

By Assumption
\ref{a:one}, $\s_a$ is not empty.  The physical
interpretation is as follows.\smallskip

\noindent
{\bf E.1.}
For all $i\in\s_a$, there is $i'\in\s_a$ such that there is a path from $(k,i)$ to
  $(k-1,i')$, avoiding level $k-1$ at the intermediary steps,
  independently of $k$.
\smallskip

\noindent
{\bf E.2.}
For all $i\in\s_a$, for all $i'\in\s_b$, there is no path from $(k,i)$ to
  $(k-1,i')$, avoiding level $k-1$ at the intermediary steps,
  independently of $k$.

Now, let us assume that $v^T K^{-1} u =0$.   This implies that 
if $v_i >0$, then $w_i =0$, so that $\s_a \subseteq \wt\s_2$, with the
possibility that $\wt\s_b$, defined as 
$\wt\s_b = \s_b \setminus(\s_1
\cup \wt\s_1) = \wt\s_2 \setminus \s_a,$
 may be empty or not empty.   In summary, we have the table

\begin{center}
  \begin{tabular}{c|cccc}
 & $\s_1$ & $\wt\s_1$ & $\wt\s_b$ & $\s_a$  \\
\hline
$u$ & $u_i >0$ & $u_i =0$ & $u_i=0$ & $u_i =0$ \\
$w$ & $w_i >0$ & $w_i >0$ & $w_i=0$ & $w_i =0$  \\
$v$   & $v_i =0$ & $v_i =0$ & $v_i=0$ & $v_i >0$
  \end{tabular}
\end{center}
where $\s_1$ and $\s_a$ are not empty, and we have
$
\s_2 = \wt\s_1 \cup \wt\s_b \cup \s_a,   ~ 
\wt\s_2 = \wt\s_b \cup \s_a,   ~
\s_b = \s_1 \cup \wt\s_1 \cup \wt\s_b
$.

Now, let us fix some arbitrary initial level $k_0$ and take any phase
$i$ in $\s_a$.
Since $\s_a \subset \wt\s_2$, we know by D.2 that any path
from $(k,i)$ to any state $(k+h,j)$ with $h \geq 0$ and $j$ in $\s_1$ must pass through level $k-1$.
By E.1 and E.2, from the state $(k,i)$, the first state $(k-1,i')$ on
any path through level $k-1$ is such that $i'$ is in $\s_a$.
Therefore, we know that any path from $(k,i)$ to any state $(k-1+h,
j)$ with $h \geq 0$ and $j$ in $\s_1$ must pass through level $k-2$.
We repeat the argument, and find that there is no path from $\ZZ
\times \s_a$ to
any state in $\ZZ \times \s_1$.   
This contradicts the property {A.1}.

\end{document}